\documentclass{svjour3steffen}

\usepackage{amssymb}
\usepackage{amsmath}
\usepackage{tikz}
\usepackage{subfig}
\usepackage{caption}
\usepackage{float}
\usepackage{mathrsfs}
\usepackage{eufrak}
\usepackage{enumerate}
\usepackage{hyperref}

 \usepackage{sidecap}
 \usepackage{graphicx}
\newcommand{\R}{\mathbb{R}}
\newcommand{\supp}{\text{supp}}
\newcommand{\spc}{\text{ }}

\begin{document}


\title{Discrete Wasserstein Barycenters:\\Optimal Transport for Discrete Data}

\author{Ethan Anderes \and Steffen Borgwardt \and Jacob Miller}

\institute{Ethan Anderes \at Department of Statistics, University of California Davis, California, U.S.A. \email{anderes@ucdavis.edu} \and
Steffen Borgwardt (corresponding author) \at Fakult\"at f\"ur Mathematik, Technische~Universit\"at M\"{u}nchen, Germany
	 \email{borgwardt@ma.tum.de}   
	 \and
Jacob Miller \at
Department of Mathematics, University of California Davis, California, U.S.A.
\email{jmiller@math.ucdavis.edu}
}









\maketitle
 \vspace*{-1cm}

\begin{abstract}

Wasserstein barycenters correspond to optimal solutions of transportation problems for several marginals, and as such have a wide range of applications ranging from economics to statistics and computer science. When the marginal probability measures are absolutely continuous (or vanish on small sets) the theory of Wasserstein barycenters is well-developed (see the seminal paper \cite{ac-11}). However, exact continuous computation of Wasserstein barycenters in this setting is tractable in only a small number of specialized cases. Moreover, in many applications data is given as a set of probability measures with finite support. In this paper, we develop theoretical results for Wasserstein barycenters in this discrete setting. Our results rely heavily on polyhedral theory which is possible due to the discrete structure of the marginals.

Our results closely mirror those in the continuous case with a few exceptions. In this discrete setting we establish that Wasserstein barycenters must also be discrete measures and there is always a barycenter which is provably sparse. Moreover, for each Wasserstein barycenter there exists a non-mass-splitting optimal transport to each of the discrete marginals. Such non-mass-splitting transports do not generally exist between two discrete measures unless special mass balance conditions hold. This makes Wasserstein barycenters in this discrete setting special in this regard. 

We illustrate the results of our discrete barycenter theory with a proof-of-concept computation for a hypothetical transportation problem with multiple marginals: distributing a fixed set of goods when the demand can take on different distributional shapes characterized by the discrete marginal distributions. A Wasserstein barycenter, in this case, represents an optimal distribution of inventory facilities which minimize the squared distance/transportation cost totaled over all demands.

\end{abstract}


\section {Introduction}

Optimal transportation problems with multiple marginals are becoming important in applications ranging from economics and finance \cite{bhp-13,ce-10,cmn-10,ght-14} to condensed matter physics and image processing \cite{bpg-12,cfk-13,jzd-98,rpdb-12,ty-05}.  The so-called Wasserstein barycenter corresponds to optimal solutions for these problems, and as such has seen a flurry of recent activity  (see \cite{ac-11,bk-12,bll-11,coo-14,cd-14,mmh-11,mtbmmh-15,p-11b,p-13,p-11,p-14,tmmh-14}).  Given probability measures $P_1,\ldots, P_N$ on $\Bbb R^d$, a Wasserstein barycenter is any probability measure $\bar P$ on $\Bbb R^d$ which satisfies
\begin{equation}
\label{two}
\sum_{i=1}^N W_2(\bar{P},P_i)^2=\inf_{P\in\mathcal P^2(\Bbb R^d)}\sum_{i=1}^N W_2( P, P_i)^2
\end{equation}
where $W_2$ denotes the quadratic Wasserstein distance and  $\mathcal P^2(\Bbb R^d)$ denotes the set of all probability measures on $\Bbb R^d$ with finite second moments. See the excellent monographs \cite{v-03,v-09} for a review of the Wasserstein metric and optimal transportation problems.

Much of the recent activity surrounding Wasserstein barycenters  stems, in part, from the seminal paper \cite{ac-11}. In that paper, Agueh and Carlier establish existence, uniqueness and an optimal transport characterization of $\bar P$ when $P_1,\ldots, P_N$ have sufficient regularity (those which vanish on small sets or which have a density with respect to Lebesgue measure). The transportation characterization of $\bar P$, in particular,  provides a theoretical connection with the solution of (\ref{two}) and the estimation of deformable templates used in medical imaging and computer vision (see \cite{jzd-98,ty-05} and references therein).
Heuristically, any measure $\bar P$ is said to be a deformable template if there exists a set of deformations   $\varphi_1,\ldots, \varphi_N$ which push-forward $\bar P$ to  $P_1,\ldots, P_N$, respectively,  and are all ``as close as possible" to the identity map.
Using a quadratic norm on the distance of each map $\varphi_1(x),\ldots, \varphi_N(x)$ to  $x$, a deformable template $\bar P$ then satisfies
\begin{equation}
\label{one}
\bar P \in \text{arg}\inf_{P\in\mathcal P^2(\Bbb R^d)}\left[ \inf_{\shortstack{\scriptsize $\{(\varphi_1,\ldots,\varphi_N) $\phantom{ s.t. }\\\scriptsize s.t. $\varphi_i( P)=P_i\}$ }} \sum_{i=1}^N \int_{\Bbb R^d}   |\varphi_i(x)-x|^2 d P(x) \right].
\end{equation}
The results of Agueh and Carlier establish that (\ref{two}) and (\ref{one}) share the same solution set when $P_1, \ldots, P_N$ have densities with respect to Lebesgue measures (for example).

\begin{figure}[t]
\begin{center}
\includegraphics[height = 2.25in]{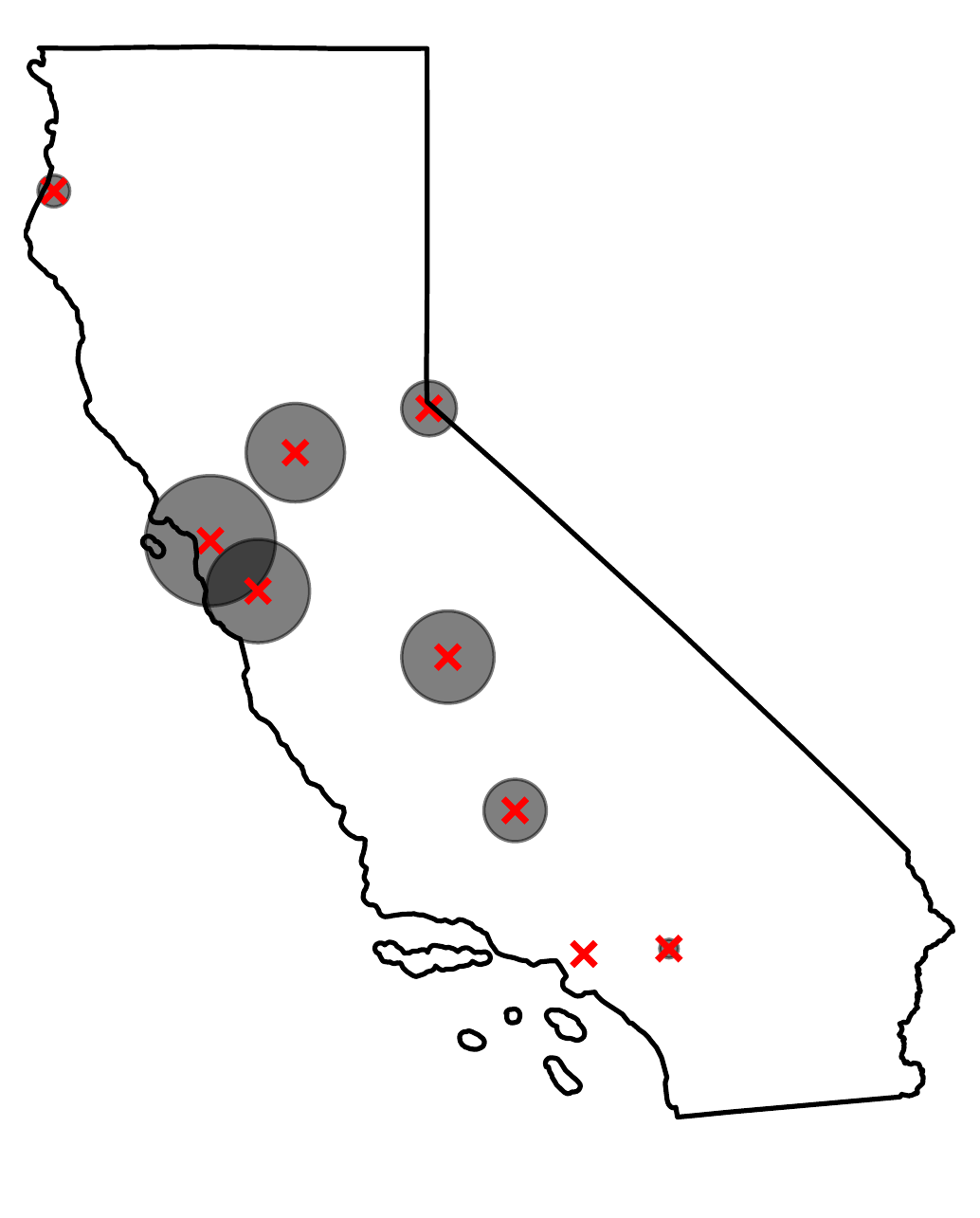}\kern-4em%
\includegraphics[height = 2.25in]{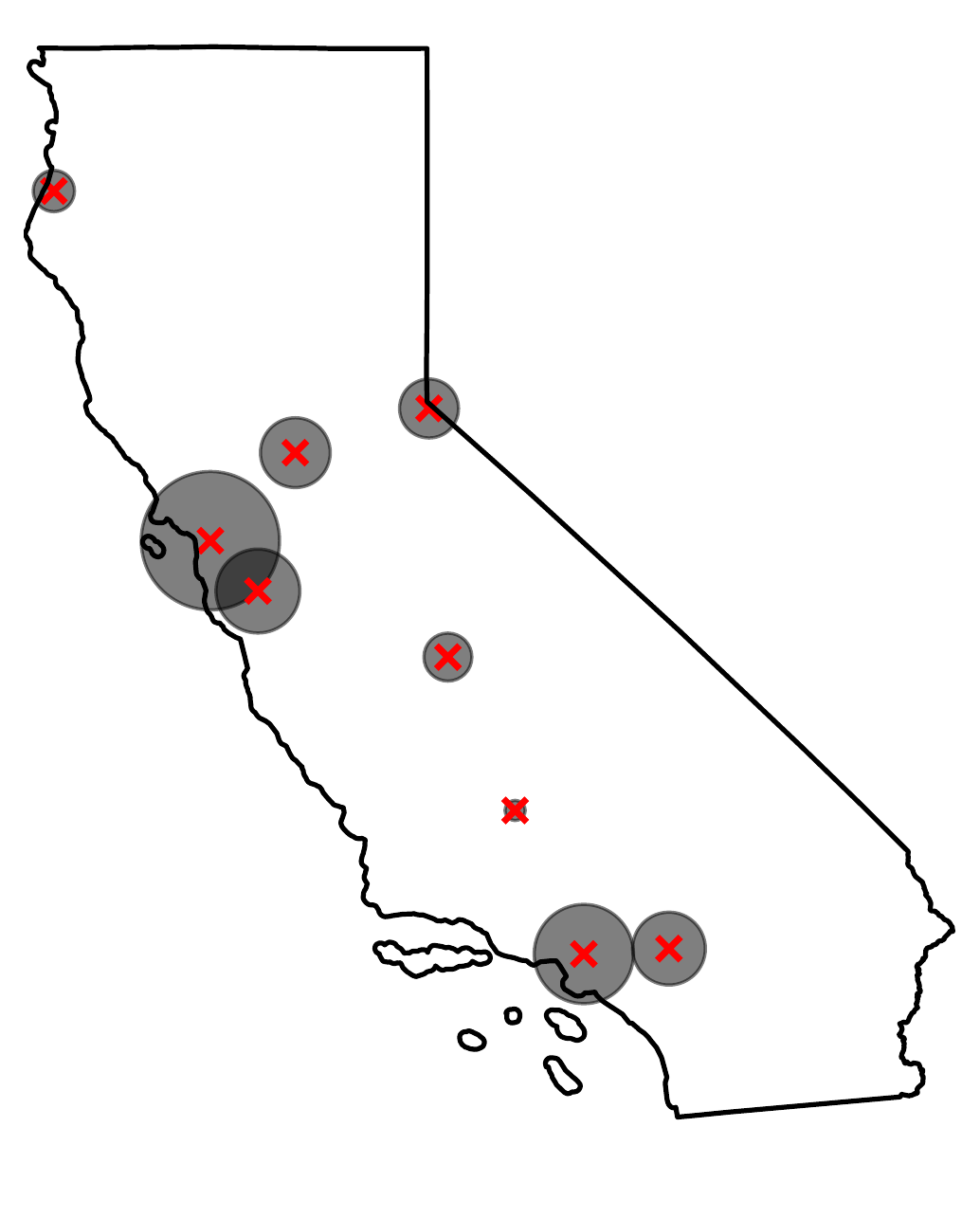}\kern-4em%
\includegraphics[height = 2.25in]{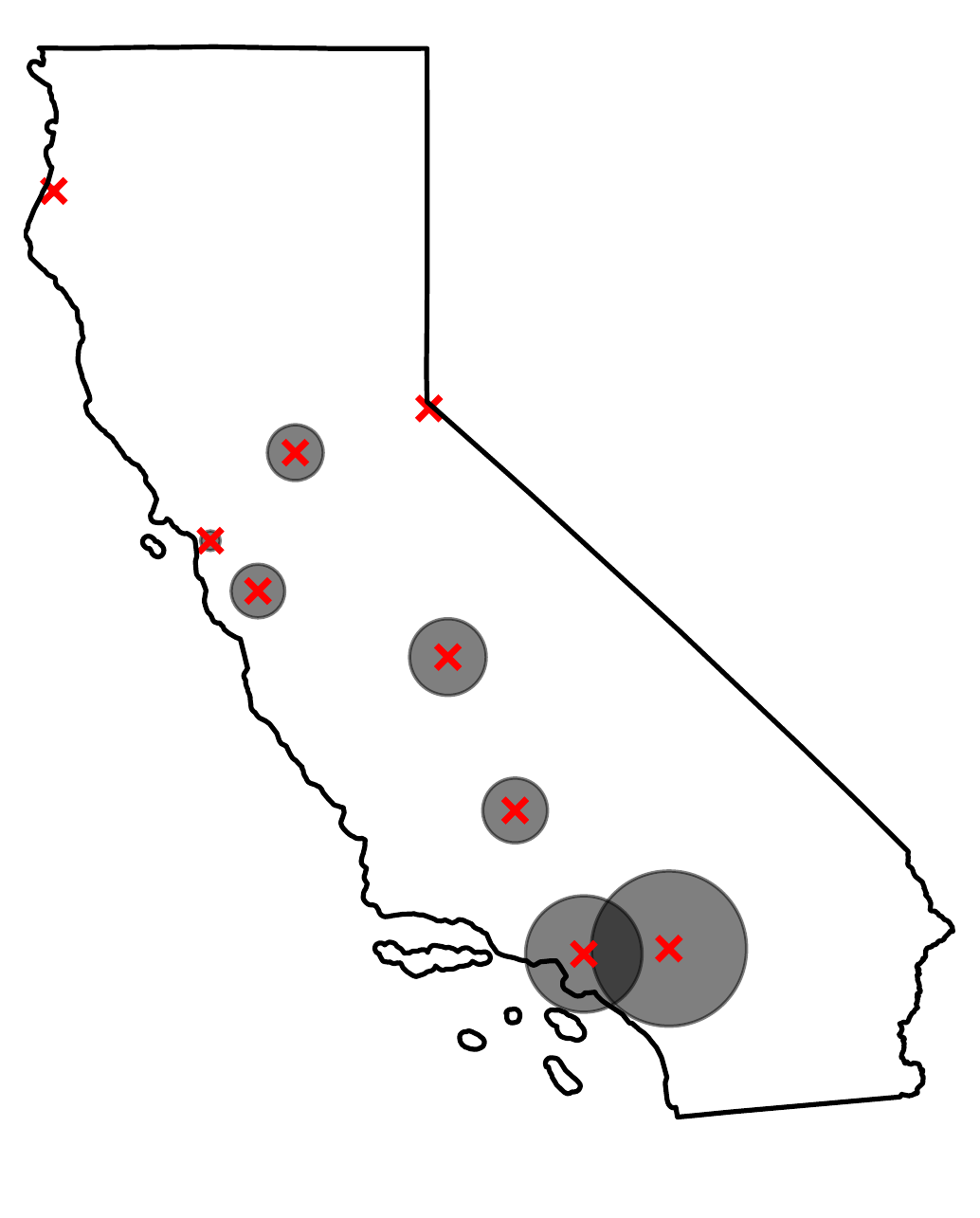}\kern-4em%
\includegraphics[height = 2.25in]{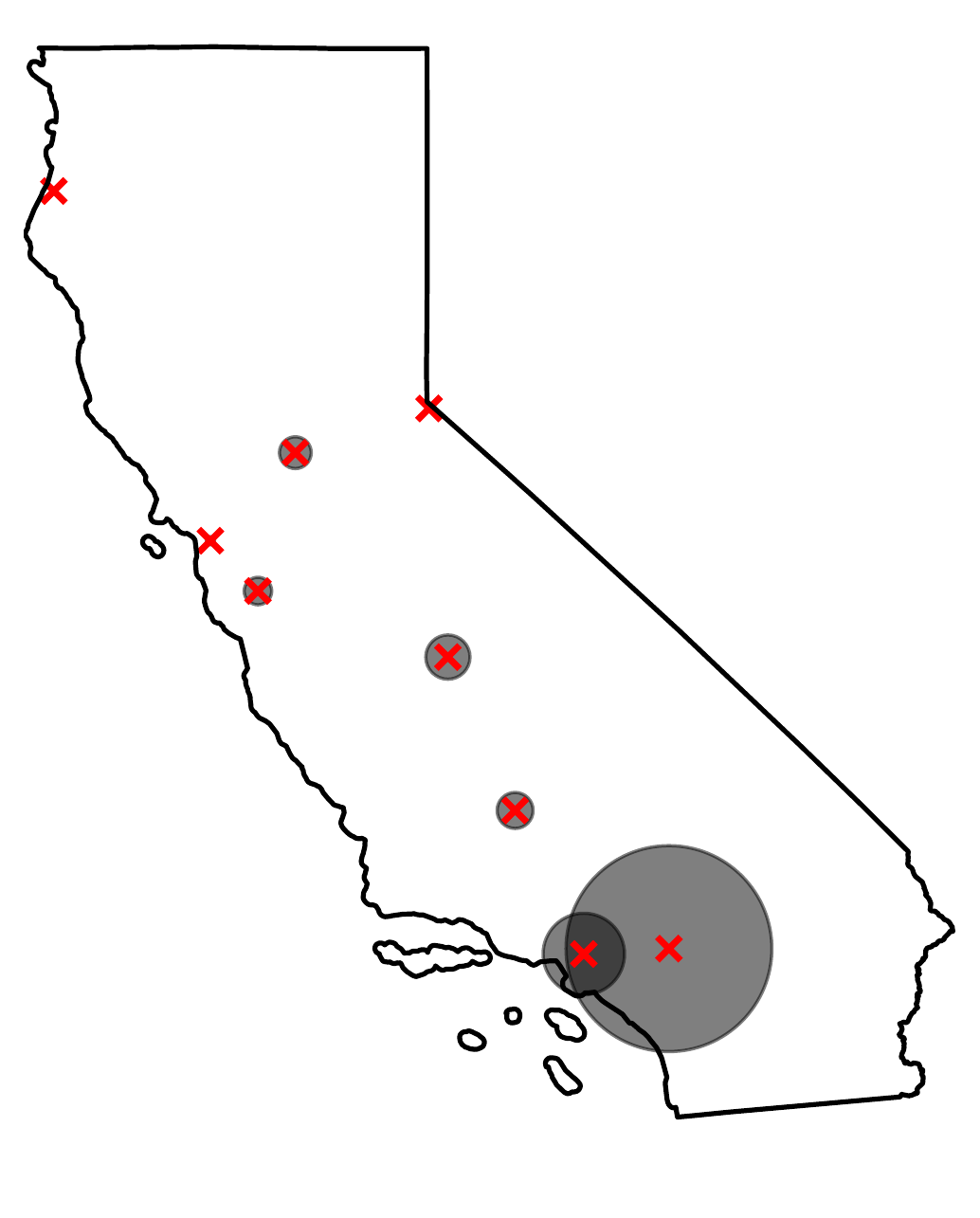}
\end{center}
\caption{\small
The above four images represent hypothetical monthly demands (as a percentage of total supply) for distributing a fixed set of goods to nine California cities (denoted by red `x' marks) in four different months (February, March, June and July). Percent demand within each month is plotted proportional to disk area and is computed from monthly average temperature and population within each city (see Section \ref{computations} for details). When percent demand is treated as a discrete probability distribution, one for each month, the Wasserstein barycenter represents the optimal distribution of inventory facilities which minimize total squared distance/transportation cost over multiple monthly demand requirements. This example serves to illustrate the applicability of the main theoretical properties derived in this paper. Theorem \ref{sparsethm}, for example, establishes that the optimal inventory distribution is a {\em sparse} discrete probability distribution with tight bounds on the scarcity of the barycenter support. In particular, the optimal inventory facilities are located at a small number of sites with relatively large storage capacity, rather than a large number small-capacity facilities distributed over a diffuse set of locations. Theorem \ref{nomasssplit} shows that the optimal transportation plan assigns each to barycenter inventory facility {\em exactly one} city to supply each month. Indeed, this type of non-mass-splitting property of optimal mass transportation is known for absolutely continuous probability distributions but does not usually hold for discrete probability distributions. The discrete Wasserstein barycenter is unique in this regard: there always exists a non-mass-splitting optimal transportation plan to each of the individual probability distributions (represented by monthly demand in this example). The Wasserstein barycenter for this example is shown in Figure \ref{barycenterFigs} and some of the optimal transportation plans are shown in Figure \ref{transportFigs}. Finally, the computational details of this example are presented in Section \ref{computations}.
}
\label{inputFigs}
\end{figure}

While absolutely continuous barycenters are mathematically interesting, in practice, data is often given as a set of {\em discrete probability measures} $P_1,\ldots, P_N$, i.e. those with finite support in $\mathbb{R}^d$.  For example, in Figure \ref{inputFigs} the discrete measures denote different demand distributions over $9$ California cities for different months (this example is analyzed in detail in Section \ref{computations}). For the remainder of the paper we refer to a {\em discrete Wasserstein barycenter} as any probability measure $\bar{ P}$ which satisfies (\ref{two}) and where all the $P_1,\ldots, P_N$ have discrete support.

In this paper we develop theoretical results for discrete Wasserstein barycenters.
Our results closely mirror those in the continuous case with a few exceptions. In the discrete case, the uniqueness and absolute continuity of the barycenter is lost. More importantly, however, is the fact that $\bar P$ is provably discrete when the marginals are discrete (see Proposition \ref{existdiscrete}). This guarantees that finite-dimensional linear programming will yield all possible optimal $\bar P$, and this in turn is utilized in this paper to study the properties of these barycenters from the point of view of polyhedral theory.  In doing so, we find remarkable differences and similarities between continuous and discrete barycenters. In particular, unlike the continuous case, there is always a discrete barycenter with provably sparse finite support; however, analogously to the continuous case, there still exists non-mass-splitting optimal transports from the discrete barycenter to each discrete marginal. Such non-mass-splitting transports generally do not exist between two discrete measures unless special mass balance conditions hold. This makes discrete barycenters special in this regard.

In Section \ref{results}, we introduce the necessary formal notation and state our main results. The corresponding proofs are found in Section $3$. To illustrate our theoretical results we provide a computational example, dicussed in Section \ref{computations} and Figures \ref{inputFigs}-\ref{transportFigs}, for a hypothetical transportation problem with multiple marginals: distributing a fixed set of goods when the demand can take on different distributional shapes characterized by $P_1, \ldots, P_N$. A Wasserstein barycenter, in this case, represents an optimal distribution of inventory facilities which minimize the squared distance/transportation cost totaled over all demands $P_1, \ldots, P_N$.

\begin{figure}
\centering
\includegraphics[height = 2.8in]{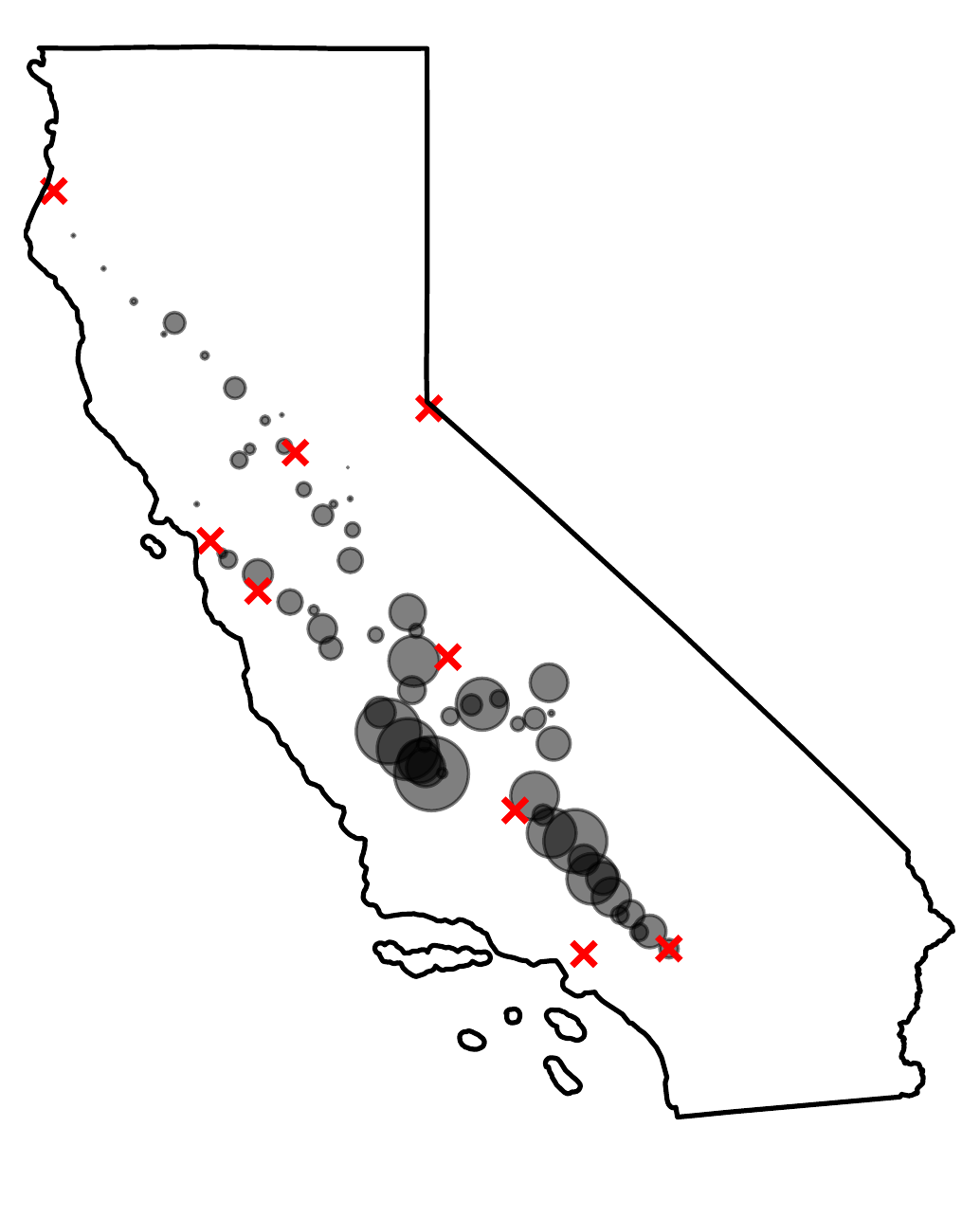}\kern-4em%
\includegraphics[height = 2.8in]{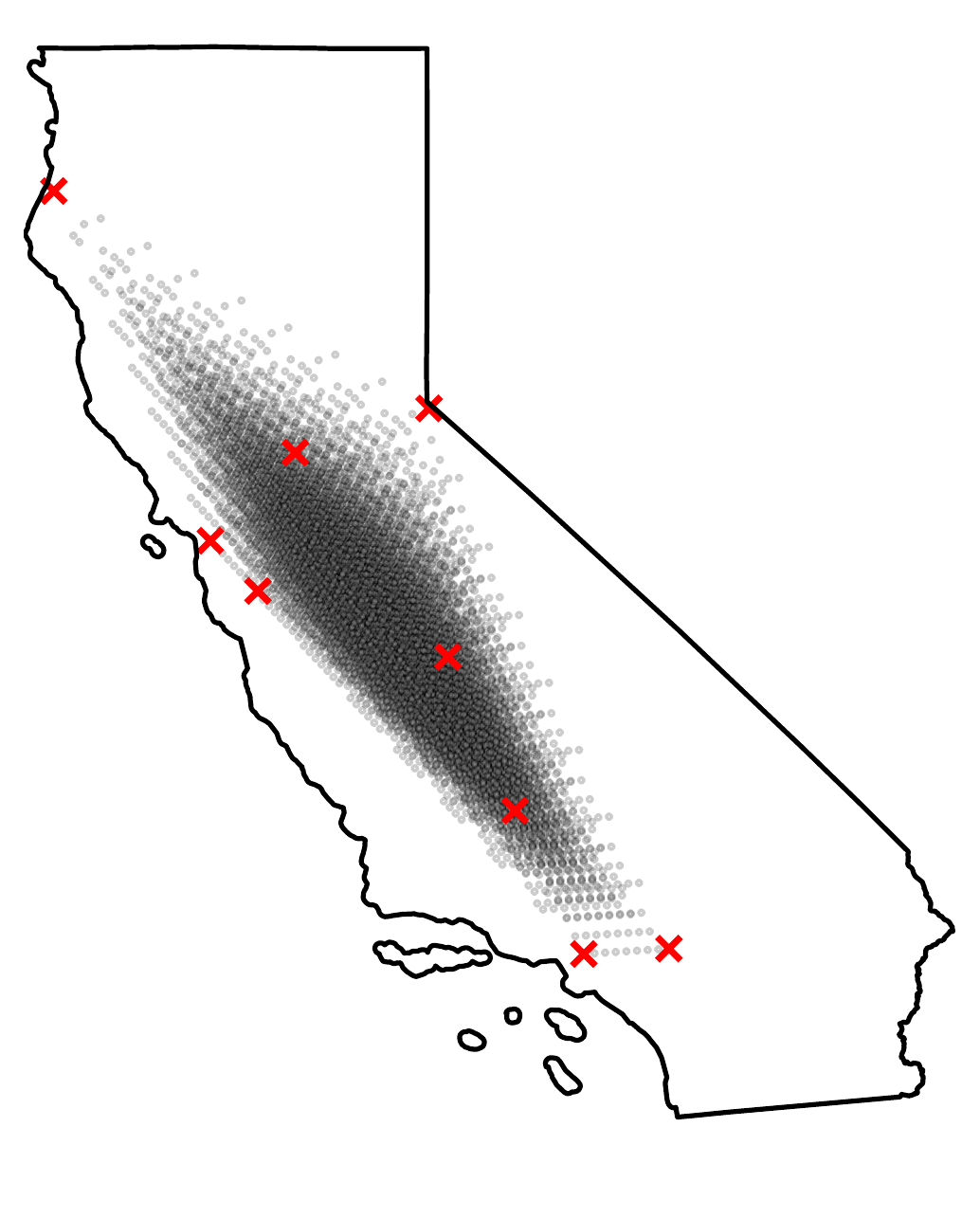}\label{barycenterFigs}
\caption{\small
The {\em leftmost image} shows a Wasserstein barycenter computed from $8$ discrete probability distributions, each representing a different monthly demand ($4$ of the months are shown in Figure \ref{inputFigs}). Notice that barycenter support is extremely sparse---supported on $63$ discrete locations---as compared to the $12870$ possible barycenter support points (shown in the {\em rightmost image}) guaranteed by Proposition \ref{existdiscrete}. Notice that Theorem \ref{sparsethm} gives an upper bound of $65$ support points for the optimal Wasserstein barycenter shown here. The role of Proposition \ref{existdiscrete}, on the other hand, is to give a finite set inclusion bound on the possible barycenter support points (shown at right in this example). This result yields the finite dimensional linear program characterization of optimal Wasserstein barycenters which is key to the analysis presented in this paper.
}
\label{barycenterFigs}
\end{figure}

\section{Results}\label{results}

For the remainder of this paper $P_1,\ldots, P_N$ will denote discrete probability measures on $\R^d$ with finite second moments. 
 Let $\mathcal{P}^2(\Bbb R^d)$ denote the space of all probability measures with finite second moments on $\R^d$. Recall, a Wasserstein barycenter $\bar P$ is an optimizer to the problem
\begin{equation}
\label{barycenter}
\inf_{P\in\mathcal P^2(\Bbb R^d)}\sum_{i=1}^N W_2( P, P_i)^2.
\end{equation}
The first important observation is that all optimizers of (\ref{barycenter}) must be supported in the finite set $S\subset \mathbb{R}^d$ where
\begin{equation}
\label{centers}
S=\left\{\frac{x_1+\ldots+x_N}{N}\big|\spc x_i\in\supp(P_i)\right\}
\end{equation}
is the set of all possible centroids coming from a combination of support points, one from each measure $P_i$. In particular, letting $\mathcal{P}_{\hspace*{-0.05cm}\mathcal{S}}^2(\Bbb R^d)=\{P\in\mathcal{P}^2(\Bbb R^d)|\spc \supp(P)\subseteq S\}$ the infinite dimensional problem (\ref{barycenter}) can be solved by replacing the requirement $P\in\mathcal P^2(\Bbb R^d)$ with $P\in\mathcal{P}_{\hspace*{-0.05cm}\mathcal{S}}^2(\Bbb R^d)$ to yield a finite dimensional minimization problem. This result follows from 
Proposition \ref{existdiscrete} below.

\begin{minipage}{\linewidth}
\begin{proposition}\label{existdiscrete} Suppose $P_1,\ldots,P_N$ are discrete probability measures on $\Bbb R^d$. Let  $\Pi({P_1,\ldots,P_N})$ denote the set of all coupled random vectors $(X_1,\ldots, X_N)$ with marginals $X_i\sim P_i$ and let $\overline X$ denote the coordinate average $\frac{X_1+\ldots+X_N}{N}$. Let $S$ be defined as in (\ref{centers}).
\begin{enumerate}[i)]
\item\label{prop_i} There exists $(X^o_1,\ldots,X_N^o)\in \Pi({P_1,\ldots,P_N})$ such that
\begin{equation}
\label{OptMean}
E\bigl|\overline{ X^o}\bigr|^2=\sup_{\shortstack{\scriptsize $ (X_1,\ldots,X_N)$\\ \scriptsize
$\quad\quad\in\Pi({P}_1,\ldots,{P}_N)$ }} E\bigl |\overline X \bigr|^2.
\end{equation}
\item\label{prop_ii}
Any $(X^o_1,\ldots,X_N^o)\in \Pi({P_1,\ldots,P_N})$ which satisfies (\ref{OptMean}) has  $\supp(\mathcal L\overline{ X^o})\subseteq S$ and
\begin{equation}
\label{7777}
\sum_{i=1}^N W_2\bigl(\mathcal L\overline{ X^o},P_i\bigr)^2=\inf_{P\in \mathcal P^2(\Bbb R^d)}\sum_{i=1}^N W_2(P,P_i)^2=\inf_{P\in \mathcal \mathcal{P}_{\hspace*{-0.05cm}\mathcal{S}}^2(\Bbb R^d)}\sum_{i=1}^N W_2(P,P_i)^2.
\end{equation}
where $\mathcal L \overline{ X^o}$ denotes the distribution (or law) of $\overline{ X^o}$.
\item\label{prop_iii} Any $\bar{P}\in\arg\min\limits_{P\in\mathcal{P}^2(\Bbb R^d)}\sum\limits_{i=1}^N W_2( P, P_i)^2$ satisfies $\supp(\bar{P})\subseteq S$.
\end{enumerate}
\end{proposition}
\end{minipage}

Notice that the existence of $(X_1^o,\ldots, X_N^o)$, in part {\em \ref{prop_i})} of the above proposition, follows immediately from the general results found in Kellerer \cite{k-84} and Rachev \cite{r-84}. Parts {\em \ref{prop_ii})} and {\em \ref{prop_iii})} are proved in Section \ref{Proof_section}. We also remark that during the preparation of this manuscript the authors became aware that Proposition \ref{existdiscrete} was independently noted in \cite{coo-14}, with a sketch of a proof. For completeness we will include a detailed proof of this statement which will also provide additional groundwork for Theorem \ref{nomasssplit}  and  Theorem  \ref{sparsethm} below.

\begin{figure}
\centering
\includegraphics[height = 2.8in]{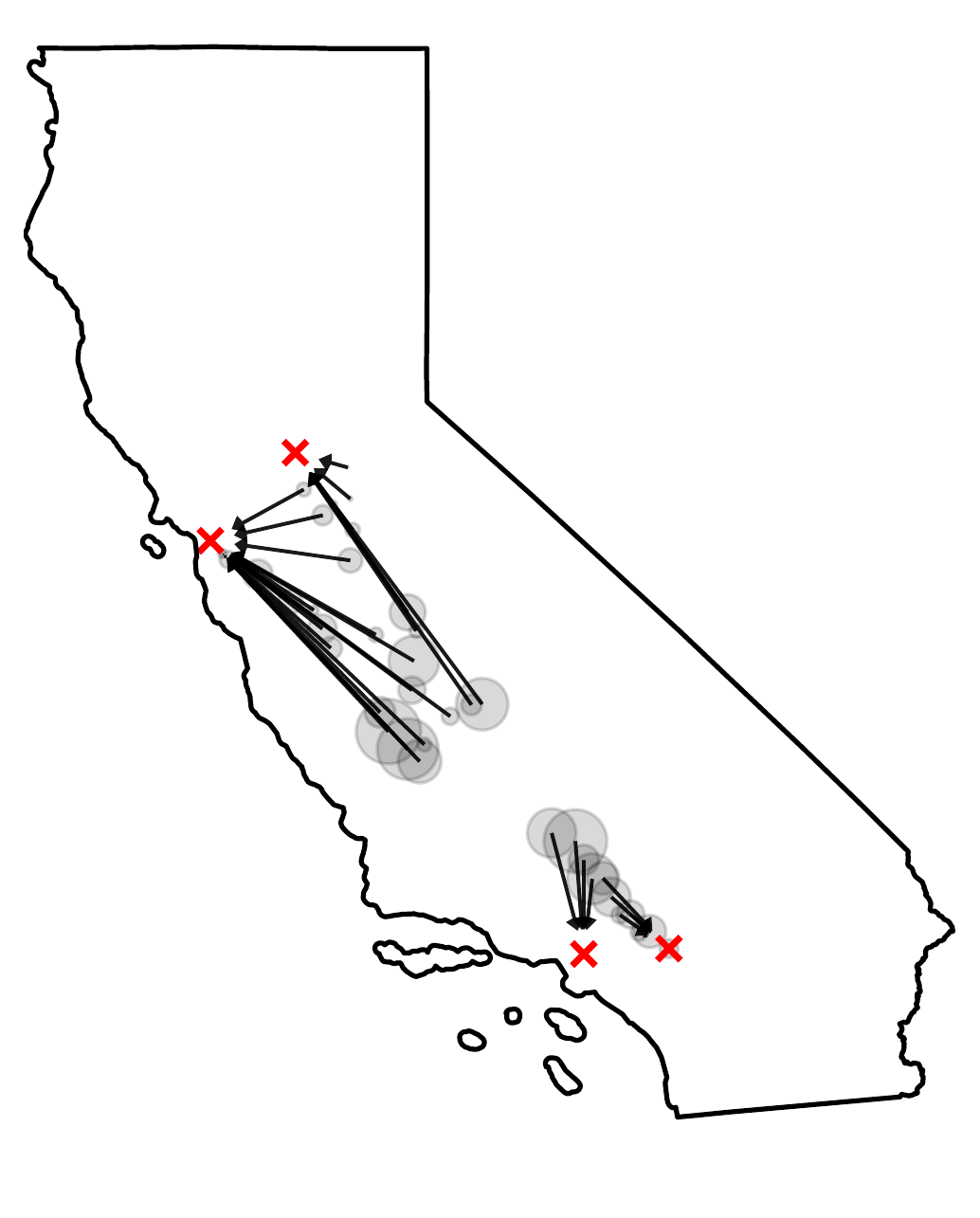}\kern-4em%
\includegraphics[height = 2.8in]{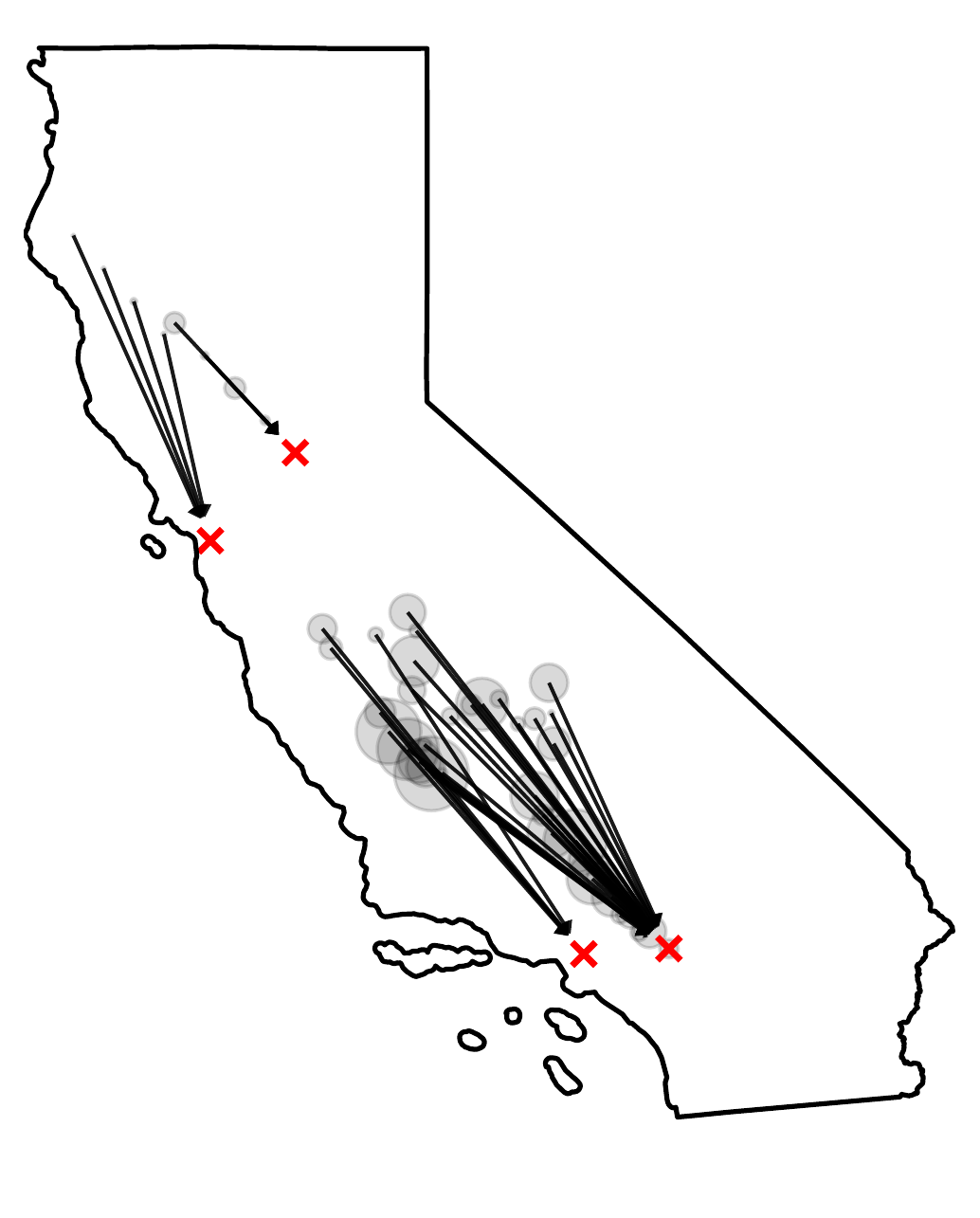}
\caption{\small
These two plots illustrate the special property of discrete Wasserstein barycenters proved in Theorem \ref{nomasssplit}: there is no mass-splitting when optimally transporting the inventory at each barycenter support to the corresponding demand for each month. The {\em image at left} shows all the transported mass flowing from the optimal barycenter into San Francisco, Sacramento, Los Angeles and San Bernardino for month of March (the corresponding March demand is shown middle-left in  Figure \ref{inputFigs}). The {\em image at right} shows the corresponding optimal transport for the month of July. Notice that these figures only show the barycenter support points which transport into the four cities shown here. The other barycenter supports transport goods to the other five cities not shown. We remark that Theorem \ref{nomasssplit} also establishes that transportation is balanced so that the transportation displacements sum to zero at each barycenter support point.
}
\label{transportFigs}
\end{figure}

Proposition \ref{existdiscrete} guarantees that any barycenter $\bar{P}$ computed with discrete marginals has the form
\begin{equation}
\bar{P}=\sum_{{\bf x}\in S}z_{\bf x}\delta_{\bf x},\hspace{0.2in}z_{\bf x}\in\R_{\geq 0}.
\end{equation}
Here $\delta_{\bf x}$ is the Dirac-$\delta$-function at ${\bf x}\in \mathbb{R}^d$ and $z_{\bf x}$ corresponds to the mass (or probability) at ${\bf x}$. 
This implies that any coupling of $\bar P$ with $P_i$, which realizes the Wasserstein distance, is in fact characterized by a finite matrix. Treating the coordinates of these matrices and the values $z_{\bf x}$ as variables, the set of all solutions to (\ref{two}) are obtained through a finite-dimensional linear program (see (\ref{LPequ}) below). In \cite{coo-14} a similar linear program was used to find approximate barycenters for sets of absolutely continuous measures by finitely approximating the support of $\bar P$ (which is sub-optimal for the continuous problem). Our use of the finite linear program characterization of $\bar P$ is different from continuous approximation. We use a version of the linear program to analyze properties of discrete barycenters themselves. Indeed, since the set of all discrete barycenters is on a face of the underlying polyhedron, one can study their properties by means of polyhedral theory.

Our first theorem illustrates a similarity between barycenters defined from absolutely continuous $P_1, \ldots, P_N$ and barycenters defined in the discrete setting. The results of \cite{ac-11} establish, in the absolutely continuous case, that there exist optimal transports from the barycenter to each $P_i$ which are optimal in the sense of Wasserstein distance and are gradients of convex functions. Theorem \ref{nomasssplit} shows that such transports not only exist for discrete barycenters but also share similar properties.

\begin{theorem}\label{nomasssplit} Suppose $P_1,\ldots,P_N$ are discrete probability measures. Let  $\bar{P}$ denote a Wasserstein barycenter solution to (\ref{two}) and let $X$ be a random variable with distribution $\bar{P}$. Then there exist finite convex functions $\psi_i:\mathbb{R}^d \rightarrow \mathbb{R}^d$, for each $i=1,\ldots, N$, such that
\begin{enumerate}[i)]
\item \label{nomasssplit_i}$\displaystyle\nabla\psi_i(\bar{P}) = P_i, \spc\forall i.$
\item \label{nomasssplit_ii}$\displaystyle E|X - \nabla\psi_i(X)|^2 = W_2(\bar{P},P_i)^2,\spc\forall i.$
\item \label{nomasssplit_iii} $\displaystyle\frac{1}{N}\sum_{i=1}^N \nabla\psi_i(x_j) = x_j,\spc\forall x_j\in\supp(\bar{P}).$
\item \label{nomasssplit_iv} $\displaystyle\frac{1}{N}\sum_{i=1}^N\psi_i(x_j) = \frac{|x_j|^2}{2},\spc\forall x_j\in\supp(\bar{P}).$
\end{enumerate}
\end{theorem}


 Intuitively, one would expect the support of a barycenter to be large to accommodate such a condition. This is particularly plausible since such these transports must realize the Wasserstein distance between each measure and the barycenter.  However, it has been noted that the barycenters of discrete measures are often sparse in practice; see for example \cite{cd-14}. Our second main result resolves this tension and establishes that there always is a  Wasserstein barycenter whose solution is theoretically guaranteed to be sparse.

\begin{theorem}\label{sparsethm} Suppose  $P_1,\ldots,P_N$ are discrete probability measures, and let $S_i=|\supp(P_i)|$. Then there exists a barycenter $\bar{P}$ of these measures such that
\begin{equation}\label{sparseequ}
|\supp(\bar{P})| \leq \sum_{i=1}^N S_i  - N + 1.
\end{equation}
\end{theorem}

\noindent We would like to stress how low this guaranteed upper bound on the size of the support of the barycenter actually is. For example, let every $P_i$ have a support of the same cardinality $T$. Then $|S| \leq T^N$ and if the support points are in general position one has $|S|=T^N$. In contrast, the support of the barycenter has cardinality at most $NT$.

Additionally, the bound in Theorem \ref{sparsethm} is the best possible in the sense that, for any natural numbers $N$ and $W$, it is easy to come up with a set of $N$ measures for which $|\supp(\bar{P})| = \sum_{i=1}^N S_i  - N + 1=W$: Choose $P_1$ to have $W$ support points and uniformly distributed mass $\frac{1}{W}$ on each of these points. Choose the other $P_i$ to have a single support point of mass $1$. Then $|S|=W$ and the barycenter uses all of these possible support points with mass $\frac{1}{W}$. 

A particularly frequent setting in applications is that all the $P_i$ are supported on the same discrete grid, uniform in all directions, in $\mathbb{R}^d$. See for example \cite{cd-14,rpdb-12} for applications in computer vision with $d=2$. In this situation, the set $S$ of possible centroids is a finer uniform grid in $\mathbb{R}^d$, which allows us to strengthen the results in Proposition \ref{existdiscrete} and Theorem \ref{sparsethm}.

\begin{corollary}\label{sparsegrid}  Let $P_1,\ldots,P_N$ be discrete probability measures supported on an $L_1{\times}\ldots{\times}L_d$-grid, uniform in all directions, in $\mathbb{R}^d$. Then there exists a barycenter $\bar{P}$ supported on a refined $(N(L_1-1)+1)\times \ldots \times (N(L_d-1)+1)$-grid, uniform in all directions, with $|\supp(\bar P)|\leq N(\prod\limits_{i=1}^d L_i-1) +1$. In particular, the density of the support of the barycenter on this finer grid is less than \begin{equation*} \frac{1}{N^{d-1}}\prod\limits_{i=1}^d\frac{ L_i}{ (L_i-1)}.
\end{equation*}
\end{corollary}

\section{Proofs}\label{Proof_section}
In this section we prove the results outlined in Section \ref{results}. We begin with a proof of Proposition \ref{existdiscrete}.

\subsection{\bf Existence of Discrete Barycenters}

Recall that a discrete barycenter $\bar{P}$ is an optimizer of $(\ref{barycenter})$ when $P_1,\ldots, P_N$ are discrete probability measures. We will show that $\bar{P}$ must have the form of a coordinatewise average of optimally coupled random vectors with marginals given by the $P_i$. In particular, we will establish the existence of $N$ random vectors $X_1^o,\ldots, X_N^o$ with marginal distributions $X_i^o\sim  P_i$ that are as highly correlated as possible so that the variability in the average $\overline{X^o}=\frac{X_1^o+\cdots+ X_N^o}{N}$ is maximized. Once these coupled random vectors  $X_1^o,\ldots, X_N^o$ are obtained, the distribution of the average $\overline{X^o}$ (denoted $\mathcal L \overline{X^o}$) will serve as $\bar{P}$.



\begin{proof}[{of Proposition \ref{existdiscrete}}]
As remarked earlier, part {\em \ref{prop_i})} of Proposition \ref{existdiscrete} follows from the general results of Kellerer \cite{k-84} and Rachev \cite{r-84}. Therefore there exists an optimally coupled random vector  $(X_1^o,\ldots, X_N^o)\in \Pi({P_1,\ldots,P_N})$  which satisfies (\ref{OptMean}). We will show that
\begin{equation}
\label{step1inprop1}
\sum_{i=1}^N W_2\bigl(\mathcal L\overline{ X^o},P_i\bigr)^2=\inf_{P\in \mathcal P^2(\Bbb R^d)}\sum_{i=1}^N W_2(P,P_i)^2.
\end{equation}
Notice the definition of $S$ automatically implies $\supp(\mathcal L\overline{ X^o})\subseteq S$ so that (\ref{step1inprop1}) will imply
\begin{equation}
\sum_{i=1}^N W_2\bigl(\mathcal L\overline{ X^o},P_i\bigr)^2=\inf_{\mathcal{P}_{\hspace*{-0.05cm}\mathcal{S}}^2(\Bbb R^d)}\sum_{i=1}^N W_2(P,P_i)^2=\inf_{P\in \mathcal P^2(\Bbb R^d)}\sum_{i=1}^N W_2(P,P_i)^2
\end{equation}
and complete the proof of part  {\em \ref{prop_ii})}. 

 So suppose $P\in \mathcal P^2(\Bbb R^d)$. Then for all $i=1,\ldots, N$ there exists an optimally coupled random vector  $(Y^*_i,X_i^*)\in \Pi(P,P_i)$ such that $ W_2(P,P_i)^2 =  E|Y^*_i-X_i^*|^2$. (This is a well known property of the Wasserstein distance $W_2$, see for example Proposition 2.1 in \cite{v-03}.) 
 Since the random variables $Y^*_1,\ldots, Y^*_N$ all have distribution $P$ it is easy to see that there exists a generalized Gluing lemma for the existence of a random vector $(Y, X_1,\ldots, X_N)\in\Pi(P,P_1,\ldots,P_N)$ such that $(Y,X_i)$ has the same distribution as $(Y^*,X_i^*)$ for each $i$. This can be seeing by first sampling a single  $Y\sim P$ then sample  $X_1,\ldots, X_N$ independently conditional on $Y$ where the conditional distribution $Pr(X_i =x|Y = y)$ is set to $Pr(X_i^* =x|Y^* = y)$ (the finite support of $P_1,\ldots, P_N$ is sufficient to guarantee existence of these conditional distributions). This yields
 \begin{equation}
 \label{rew}
  \frac{1}{N} \sum_{i=1}^N W_2(P,P_i)^2= \frac{1}{N}\sum_{i=1}^N E |Y_i^*-X^*_i|^2= \frac{1}{N}\sum_{i=1}^N E |Y-X_i|^2.
  \end{equation}
Now note that $X_i\sim P_i$ and $X_i^o\sim P_i$. Thus
\begin{align}
\label{hugeThing}
 \sum_{i=1}^N E\bigl |\overline{X^o}- X^o_i \bigr|^2&=\sum_{i=1}^NE\bigl|\overline{X^o}\bigr|^2-2E\sum_{i=1}^N\langle \overline{X^o},X^o_i\rangle +\sum_{i=1}^NE\bigl|X^o_i\bigl|^2=-NE\bigl|\overline{X^o}\bigr|^2 + \sum_{i=1}^NE\bigl|X^o_i\bigl|^2\nonumber\\
&= -NE\bigl|\overline{X^o}\bigr|^2 + \sum_{i=1}^NE\bigl|X_i\bigl|^2 = \inf_{\shortstack{\scriptsize $ (X_1,\ldots,X_N)$\\ \scriptsize
$\quad\quad\in\Pi({P}_1,\ldots,{P}_N)$ }} -NE\bigl |\overline X \bigr|^2 + \sum_{i=1}^NE\bigl|X_i\bigl|^2\nonumber \\
&=\inf_{\shortstack{\scriptsize $ (X_1,\ldots,X_N)$\\ \scriptsize
$\quad\quad\in\Pi({P}_1,\ldots,{P}_N)$ }} \sum_{i=1}^NE\bigl|\overline{X}\bigr|^2-2E\sum_{i=1}^N\langle \overline{X},X_i\rangle +\sum_{i=1}^NE\bigl|X_i\bigl|^2\nonumber\\
&=\inf_{\shortstack{\scriptsize $ (X_1,\ldots,X_N)$\\ \scriptsize
$\quad\quad\in\Pi(P_1,\ldots,P_N)$ }}  \sum_{i=1}^N E\bigl | \overline X- X_i \bigr|^2.
\end{align}
Also, note that
\begin{equation}\label{smallThing}E|\overline{ X^o}-X^o_i|^2\geq \inf_{ (Y,X) \in\Pi(\mathcal L\overline{ X^o},P_i)}
  E|Y-X|^2=W_2(\mathcal L\overline{ X^o},P_i)^2.\end{equation}
Combining $(\ref{hugeThing})$ and $(\ref{smallThing})$, we get
\begin{equation}\label{optdist}
\frac{1}{N}\sum_{i=1}^N E |\overline X-X_i|^2 \geq \frac{1}{N}\sum_{i=1}^N E|\overline{ X^o}-X^o_i|^2 \geq  \frac{1}{N}\sum_{i=1}^N W_2(\mathcal L\overline{ X^o},P_i)^2.
\end{equation}
Further we have a minorant for the right hand side of (\ref{rew}) as follows
 \begin{align}
 \frac{1}{N}\sum_{i=1}^N E |Y-X_i|^2 &=\frac{1}{N}\sum_{i=1}^N E |Y-\overline{X} + \overline{X}-X_i|^2\nonumber\\
&=\frac{1}{N}\sum_{i=1}^N E |Y-\overline{X}|^2 +\frac{2}{N}E\sum_{i=1}^N\langle Y-\overline{X},\overline{X} - X_i\rangle +\frac{1}{N} \sum_{i=1}^N E|\overline X-X_i|^2 \nonumber\\
&= E |Y-\overline{X}|^2 +\frac{2}{N}E\langle Y-\overline{X},\sum_{i=1}^N(\overline{X} - X_i)\rangle +\frac{1}{N}\sum_{i=1}^NE|\overline X-X_i|^2\nonumber\\
&=E |Y-\overline{X}|^2 + \frac{1}{N}\sum_{i=1}^N E|\overline X-X_i|^2 \geq \frac{1}{N}\sum_{i=1}^N E|\overline X-X_i|^2.\label{minor}
  \end{align}
 Putting (\ref{rew}), (\ref{optdist}), and (\ref{minor}) together we obtain
\begin{equation}
\frac{1}{N} \sum_{i=1}^N W_2(P,P_i)^2 = \frac{1}{N}\sum_{i=1}^N E |Y-X_i|^2 \geq \frac{1}{N}\sum_{i=1}^N E|\overline X-X_i|^2\geq \frac{1}{N}\sum_{i=1}^N W_2(\mathcal L\overline{ X^o},P_i)^2.
\end{equation}
This shows that $\mathcal L\overline{ X^o}$ is a minimizer of our problem and hence a barycenter, proving part {\em \ref{prop_i})}.

Finally, to prove part {\em \ref{prop_iii})}, note that if $P\in \mathcal P^2(\Bbb R^d)$ and $\supp(P)\nsubseteq S$, then any coupling $(Y, X_1,\ldots, X_N)\in\Pi(P,P_1,\ldots,P_N)$
must satisfy $E|Y-\overline{X}|^2>0$ (since  $\supp(\overline{X})\subseteq S$ and $\supp(P)\nsubseteq S$). This implies, by the last line of (\ref{minor}), that
\begin{equation}\frac{1}{N}\sum_{i=1}^N E |Y-X_i|^2 > \frac{1}{N}\sum_{i=1}^N E|\overline X-X_i|^2,\end{equation}
and hence that
\begin{equation}\frac{1}{N} \sum_{i=1}^N W_2(P,P_i)^2 = \frac{1}{N}\sum_{i=1}^N E |Y-X_i|^2 >\frac{1}{N} \sum_{i=1}^N E|\overline X-X_i|^2\geq \frac{1}{N}\sum_{i=1}^N W_2(\mathcal L\overline{ X^o},P_i)^2,\end{equation}
so that $P$ is not a barycenter. Therefore for any barycenter $\bar{P}$, we must have $\supp(\bar{P})\subseteq S$, which proves part {\em \ref{prop_iii})}. \qed
\end{proof}

\subsection{\bf Linear Programming and Optimal Transport}

Let us now develop a linear programming model (LP) for the exact computation of a discrete barycenter. Suppose we have a set of discrete measures $P_i$, $i=1,\ldots,N$, and additionally another discrete measure $P$. Let $S_0=|\supp(P)|$ and $S_i=|\supp(P_i)|$ for each $i$ as before. Let $x_{j}$, $j=1,\ldots,S_0$ be the points in the support of $P$, each with mass $d_{j}$, and let $x_{ik}$, $k=1,\ldots,S_i$ be the points in the support of $P_i$, each with mass $d_{ik}$. For the sake of a simple notation in the following, when summing over these values, the indices take the full range unless stated otherwise.

If $(X,Y_i)\in\Pi(P,P_i)$, then this coupling can be viewed as a finite matrix, since both probability measures are discrete. We define $y_{ijk}\geq 0$ to be the value of the entry corresponding to the margins $x_j$ and $x_{ik}$ in this finite matrix.

Note in this coupling that $\sum_k y_{ijk} = d_j$ for all $j$ and that $\sum_j y_{ijk} = d_{ik}$ for all $k$ and further that
\begin{equation}\label{distcost}
E|X-Y|^2 = \sum_{j,k} |x_j - x_{ik}|^2\cdot y_{ijk} = \sum_{j,k} c_{ijk}\cdot y_{ijk},
\end{equation}
where $c_{ijk}:= |x_j - x_{ik}|^2$ just by definition.

Given a non-negative vector ${\bf y}= (y_{ijk})\geq 0$ that satisfies $\sum_k y_{ijk} = d_j$ for all $i$ and $j$ and $\sum_j y_{ijk} = d_{ik}$ for all $i$ and $k$, we call ${\bf y}$ an \emph{N-star transport} between $P$ and the $P_i$. We define the \emph{cost} of this transport to be $c({\bf y}):=\sum_{i,j,k} c_{ijk}\cdot y_{ijk}$.

Further there exist vectors $(X^*,Y_i^*)\in\Pi(P,P_i)$ for all $i$, and a corresponding $N$-star transport ${\bf y^*}$, such that
\begin{equation}
\sum_{i} W_2(P,P_i)^2 = \sum_i E|X^*-Y^*|^2 = c({\bf y^*}).
\end{equation}
For any $(X,Y_i)\in\Pi(P,P_i)$ we also have $E|X^*-Y_i^*|^2\leq E|X-Y_i|^2$, and hence it is easily seen that ${\bf y^*}$ is an optimizer to the following linear program
\begin{align}\label{nstar}
\min_{\bf y}  &\spc\spc c({\bf y}) \nonumber \\
\sum_k y_{ijk}  = &\spc\spc d_j,  \spc\spc\forall i=1,\ldots,N,\spc\forall j=1,\ldots,S_0,\nonumber\\
\sum_{j} y_{ijk}  = &\spc\spc d_{ik}, \spc\spc\forall i=1,\ldots,N,\spc\forall k=1,\ldots,S_i,\\
y_{ijk}  \geq & \spc\spc0, \spc\spc\spc\spc\spc\forall i=1,\ldots,N,\spc\forall j=1,\ldots,S_0,\spc\forall k=1,\ldots,S_i.\nonumber
\end{align}

Now suppose we wish to find a barycenter using a linear program. Then using Proposition \ref{existdiscrete} we know that this amounts to finding a solution to
\begin{equation}
\min_{P\in\mathcal{P}_{\hspace*{-0.05cm}\mathcal{S}}^2(\Bbb R^d)}\sum_{i=1}^N W_2( P, P_i)^2,\hspace{0.5in}P =\sum_{{\bf x}\in S}z_{\bf x}\delta_{\bf x},\hspace{0.2in}z_{\bf x}\in\R_{\geq 0}.
\end{equation}
Using this we can expand the possible support of $P$ in the previous LP to $S$, and let the mass at each $x_j\in S$ be represented by a variable $z_j\geq 0$. This is a probability distribution if and only if the constraint $\sum_j z_j =1$ is satisfied. Then every exact barycenter, up to measure-zero sets, must be represented by some assignment of these variables and hence is an optimizer of the LP
\begin{align}
\min_{\bf y,z}  &\spc\spc c({\bf y}) \nonumber \\
\sum_k y_{ijk}  = &\spc\spc z_j, \spc\spc\forall i=1,\ldots,N,\spc\forall j=1,\ldots,S_0,\nonumber\\
\sum_{j} y_{ijk}  = &\spc\spc d_{ik}, \spc\spc\forall i=1,\ldots,N,\spc\forall k=1,\ldots,S_i,\nonumber\\
y_{ijk}  \geq &  \spc\spc0,\spc\spc\spc \spc\spc\forall i=1,\ldots,N,\spc\forall j=1,\ldots,S_0,\spc\forall k=1,\ldots,S_i,\nonumber\\
z_j \geq &  \spc\spc0,\spc\spc\spc \spc\spc\forall j=1,\ldots,S_0.\label{LPequ}
\end{align}
Since each $P_i$ is a probability distribution it is easy to see that $\sum_j z_j =1$ is just a consequence of satisfaction of the other constraints. Any optimizer $({\bf y^*},{\bf z^*})$ to this LP is a barycenter $\bar{P}$ in that
\begin{equation}
\min_{P\in\mathcal{P}_{\hspace*{-0.05cm}\mathcal{S}}^2(\Bbb R^d)}\sum_{i=1}^N W_2( P, P_i)^2=\sum_{i=1}^N W_2(\bar{P}, P_i)^2=c({\bf y^*})\hspace{0.3in}\text{ and }\hspace{0.3in}\bar{P} =\sum_{j}z^*_j\delta_{x_j}.
\end{equation}

It is notable that the LP in $(\ref{LPequ})$ corresponds to $N$ transportation problems, linked together with variables $z_j$, representing a common marginal for each transportation problem. In fact it is not hard to show that in the case $N=2$ this LP can be replaced with a network flow LP on a directed graph. It is easily seen that this LP is both bounded (it is a minimization of a positive linear sum of non-negative variables) and feasible (assign an arbitary $z_j=1$ and the remainder of them $0$ and this reduces to solving $N$ transportation LPs). Thus it becomes useful to write down the dual LP, which also bares similarity to a dual transportation problem
\begin{align}
\max_{\bf \tau,\theta}  &\spc\spc \sum_{i,k}d_{ik}\cdot \tau_{ik}\nonumber \\
\theta_{ij} + \tau_{ik} \leq&\spc\spc c_{ijk},\spc\spc\forall i=1,\ldots,N,\spc\forall j=1,\ldots,S_0,\spc\forall k=1,\ldots,S_i,\nonumber\\
\sum_{j} \theta_{ij}  \geq &\spc\spc 0, \spc\spc\spc\spc\spc\forall i=1,\ldots,N,\spc\forall j=1,\ldots,S_0,\label{dualLP}
\end{align}
where there is a variable $\tau_{ik}$ for each defining measure $i$ and each $x_{ik}\in\supp(P_i)$ and a variable $\theta_{ij}$ for each defining measure $i$ and each $x_j\in S$.

These LPs not only will be used for computations in Section \ref{computations}, but also can be used to develop the necessary theory for Theorem \ref{nomasssplit}.

\begin{lemma}\label{nosplitlem} Let $P_1,\ldots,P_N$ be discrete probability measures with a barycenter $\bar{P}$ given by a solution $({\bf y^*},{\bf z^*})$  to (\ref{LPequ}). Then
\begin{enumerate}[i)]
\item\label{nosplitlem_i} For any $x_j\in\supp(\bar{P})$ (i.e. $z_j^*>0$) combined with any choice of $x_{ik_i}\in\supp(P_i)$ for $i=1,\ldots,N$ such that $y^*_{ijk_i}>0$ for each $i$, one then has $x_j=\frac{1}{N}\sum_i x_{ik_i}$.
\item\label{nosplitlem_ii} For any $x_j\in\supp(\bar{P})$ and $i=1,\ldots,N$, one has $\big|\;\{y^*_{ijk}>0|\spc x_{ik}\in\supp(P_i)\}\;\big| = 1$.
\end{enumerate}
\end{lemma}

\begin{proof} \emph{i)} Suppose the statement in \emph{i)} is false. Then there exists an $x_{j_0}\in\supp(\bar{P})$ and there are points $x_{ik_i}\in\supp(P_i)$ for $i=1,\ldots,N$ such that $y^*_{ij_0k_i}>0$ for each $i$ and $x_{j_0}\neq\frac{1}{N}\sum_i x_{ik_i}$.

Let $\alpha = \min_i y^*_{ij_0k_i} >0$ and let $x_{j^*} = \frac{1}{N}\sum_i x_{ik_i}$. Then define $({\bf \hat{y}},{\bf \hat{z}})$ such that $\hat{y}_{ij_0k_i} = y^*_{ij_0k_i} - \alpha$ for each $i$, $\hat{y}_{ij^*k_i} = y^*_{ij^*k_i} + \alpha$ for each $i$, $\hat{z}_{j_0} = z^*_{j_0} - \alpha$, $\hat{z}_{j^*} = z^*_{j^*} + \alpha$, and $\hat{z}_j = z^*_j$ and $\hat{y}_{ijk}=y^*_{ijk}$ for all other variables.

It is easily checked that $({\bf \hat{y}},{\bf \hat{z}})$ is also a feasible solution to (\ref{LPequ}). Further
\begin{equation}
c({\bf \hat{y}}) = c({\bf y^*}) + \alpha\left( \sum_i c_{ij^*k_i} - \sum_i c_{ij_0k_i}\right)<c({\bf y^*}),
\end{equation}
where the strict inequality follows since $x_{j_0}\neq\frac{1}{N}\sum_i x_{ik_i}=x_{j^*}$ and therefore
\begin{equation}
\sum_i c_{ij_0k_i} = \sum_i |x_{j_0} - x_{ik_i}|^2 > \sum_i |x_{j^*} - x_{ik_i}|^2 =  \sum_i c_{ij^*k_i},
\end{equation}
which is a contradiction with $\bar{P}$ being a barycenter.

\emph{ii)} If $x_j\in\supp(\bar{P})$, then $z^*_j>0$ and therefore $\big|\;\{y^*_{ijk}>0|\spc x_{ik}\in\supp(P_i)\}\;\big| \geq 1$ for all $i$ is an immediate consequence of the contraints in (\ref{LPequ}). Suppose rhe statement is false, then there is some $x_j\in\supp(\bar{P})$ such that, without loss of generality, $\big|\{y^*_{1jk}>0|\spc x_{1k}\in\supp(P_1)\}\big| \geq 2$. Then we can choose $x_{1k'}\neq x_{1k''}$ such that $y^*_{1jk'},y^*_{1jk''}>0$ and further can choose $x_{ik_i}$ for $i=2,\ldots,N$ such that $y_{ijk_i}>0$ for each $i$. Then this implies, by part $(i)$, that
\begin{equation}
\frac{1}{N}\bigl(x_{1k'} + \sum_{i=2}^Nx_{ik_i}\bigr) = x_j = \frac{1}{N}\bigl(x_{1k''} + \sum_{i=2}^Nx_{ik_i}\bigr),
\end{equation}
which in turn immediately would imply $x_{1k'} = x_{1k''}$; a contradiction with our choice of $x_{1k'}\neq x_{1k''}$. Hence $\big|\;\{y^*_{1jk}>0|\spc x_{ik}\in\supp(P_1)\}\;\big|=1$.  \qed
\end{proof}

Lemma \ref{nosplitlem} already implies that there exists a transport from any barycenter $\bar{P}$ to each $P_i$. However, to prove Theorem \ref{nomasssplit} we need the concept of {\em strict complimentary slackness}. If you have a primal LP $\{\min {\bf c}^T{\bf x}|\spc {\bf A}{\bf x}={\bf b},\spc {\bf x}\geq{\bf 0}\}$ which is bounded and feasible and its dual LP $\{\max {\bf b}^T{\bf y}|\spc {\bf A}^T{\bf y}\leq{\bf c}\}$, then complimentary slackness states that the tuple $({\bf x^*},{\bf y^*})$ gives optimizers for both of these problems if and only if $x^*_i(c_i-{\bf a_i}^T{\bf y^*}) = 0$ for all $i$, where ${\bf a_i}$ is the $i$-th column of ${\bf A}$. This statement can be strengthened in form of the strict complimentary slackness condition \cite{z-94}:

\begin{proposition}\label{strictcomp} Given a primal LP $\{\min {\bf c}^T{\bf x}|\spc {\bf A}{\bf x}={\bf b},\spc {\bf x}\geq{\bf 0}\}$ and the corresponding dual LP\\  ${\{\max {\bf b}^T{\bf y}|\spc {\bf A}^T{\bf y} \leq{\bf c}\}}$, both bounded and feasible, there exists a tuple of optimal solutions $({\bf x^*},{\bf y^*})$, to the primal and dual respectively, such that for all $i$
\begin{equation}
x^*_i(c_i-{\bf a_i}^T{\bf y^*}) = 0, \hspace{1in} x^*_i + (c_i-{\bf a_i}^T{\bf y^*})  >0.
\end{equation}
\end{proposition}
With these tools, we are now ready to prove Theorem  \ref{nomasssplit}.

\begin{proof}[{Proof of Theorem \ref{nomasssplit}}] Let $({\bf y^*},{\bf z^*},{\bf \tau^*},{\bf \theta^*})$ be a solution to (\ref{LPequ}) and (\ref{dualLP}), as guaranteed by Proposition \ref{strictcomp}. Let $\bar{P}$ be a barycenter corresponding to the solution $(\hat{\bf y},\hat{\bf z})$. For each $x_j\in\supp(\bar{P})$ let $x_{ik_j}\in\supp(P_i)$ be the unique location such that $\hat{y}_{ijk_j}>0$ as guaranteed by Lemma \ref{nosplitlem} part {\em \ref{nosplitlem_ii})}.
Now for each $i$ define
\begin{equation}
\psi_i(x) = \max_{x_{ik}\in\supp(P_i)}\langle x,x_{ik}\rangle - \frac{1}{2}|x_{ik}|^2 + \frac{1}{2}\tau^*_{ik}.
\end{equation}
Using Lemma \ref{nosplitlem} part {\em \ref{nosplitlem_i})}, it is easy to see that for proving part {\em \ref{nomasssplit_i})}-{\em \ref{nomasssplit_iii})} of  Theorem \ref{nomasssplit} it suffices to show that for each $\psi_i$ we have that $\nabla\psi_i(x_j)=x_{ik_j}$ for each $x_j\in\supp(\bar P)$.

By definition, each $\psi_i$ is convex (as the maximum over a set of linear functions) and $\psi_i(x)$ is finite for all $x\in \mathbb{R}^d$. Further \begin{align}
|x|^2 - 2\psi_i(x)& =  |x|^2 - 2\max_{x_{ik}\in\supp(P_i)}\langle x,x_{ik}\rangle - \frac{1}{2}|x_{ik}|^2 + \frac{1}{2}\tau^*_{ik}\nonumber\\
& = \min_{x_{ik}\in\supp(P_i)} |x|^2 -2\langle x,x_{ik}\rangle +|x_{ik}|^2 - \tau^*_{ik}\nonumber\\
& = \min_{x_{ik}\in\supp(P_i)} |x-x_{ik}|^2 - \tau^*_{ik},
\end{align}
and hence
\begin{equation}
|x_j|^2 - 2\psi_i(x_j)=\min_{x_{ik}\in\supp(P_i)} |x_j-x_{ik}|^2 - \tau^*_{ik}=\min_{x_{ik}\in\supp(P_i)} c_{ijk} - \tau^*_{ik}.
\end{equation}
By complimentary slackness, we have that since $\hat{y}_{ijk_j}\neq 0$, that $c_{ijk_j} -\tau^*_{ik_j}-\theta^*_{ij}=0$. Therefore by strict complimentary slackness we get  $y^*_{ijk_j}\neq 0$ and hence by Lemma \ref{nosplitlem} part {\em \ref{nosplitlem_ii})} we get $y^*_{ijk}=0$ for all $k\neq k_j$. This implies by strict complimentary slackness that for all $k\neq k_j$ we obtain $c_{ijk} -\tau^*_{ik}-\theta^*_{ij}\neq 0$ and therefore, by feasibility, that $c_{ijk}-\tau^*_{ik}<\theta^*_{ij}$. Factoring in that $c_{ijk_j} -\tau^*_{ik_j} = \theta^*_{ij}$ by complimentary slackness we have that $|x_j|^2 - 2\psi_i(x_j) = \theta^*_{ij}$. Further, since the function corresponding to $k_j$ is the only continuous function in the minimization that achieves this minimum at $x_j$ (by the above argument), we obtain that for $x$ in some neighborhood of $x_j$
\begin{align}
\;\; & |x|^2 - 2\psi_i(x) = |x-x_{ik_j}|^2 - \tau^*_{ik_j},\nonumber\\
\Rightarrow\;\; & \psi_i(x)  =\langle x,x_{ik_j}\rangle - \frac{1}{2}|x_{ik_j}|^2 + \frac{1}{2}\tau^*_{ik_j},\nonumber\\
\Rightarrow\;\; & \nabla\psi_i(x) = x_{ik_j},\\
\end{align}
so that $\nabla\psi_i(x_j)=x_{ik_j}$. Further, note that complimentary slackness implies $\sum_i\theta^*_{ij} = 0$ for each $x_j\in\supp(\bar{P})$ and hence
\begin{equation}
0 = \sum_i\theta^*_{ij} = \sum_i |x_j|^2 - 2\psi_i(x_j)
\end{equation}
\begin{equation}
\Rightarrow\frac{1}{N} \sum_i \psi_i(x_j) = \frac{|x_j|^2}{2}.
\end{equation}
This shows part {\em \ref{nomasssplit_iv})} of Theorem \ref{nomasssplit} and thus completes the proof.  \qed
\end{proof}

\subsection{\bf Sparsity and Transportation Schemes}

As before, let $P_1,\ldots,P_N$ be discrete probability measures, with point masses $d_{ik}$ for $x_{ik}\in \supp(P_i)$ defined as in the previous subsection. Then for any set $\mathcal{S}\subseteq S\times \supp(P_1)\times\ldots\times\supp(P_N)$ we fix an arbitary order on $\mathcal{S}$, i.e. $\mathcal{S}=\{s_1,s_2,\dots,s_m\}$ where each $s_h=(q_{h0},q_{h1},\dots,q_{hN})$, and define a \emph{location-fixed transportation scheme} as the set
\begin{equation}
\mathcal{T}(\mathcal{S}) := \{{\bf w}\in\R_{\geq 0}^{|\mathcal{S}|}| \sum_{\substack{h=1,\\ q_{hi}=x_{ik}}}^m w_h = d_{ik},\spc\forall i=1,\dots,N,\spc\forall k=1,\dots,S_i\}.
\end{equation}
Informally, the coefficients of ${\bf w}\in\mathcal{T}(\mathcal{S})$ correspond to an amount of transported mass from a given location in $S$ to combinations of support points in the $P_i$, where each of these support points receives the correct total amount. Given a ${\bf w}$, we define its corresponding discrete probability measure
\begin{equation}
P({\bf w},\mathcal{S}):=\sum_{h=1}^m w_h\delta_{q_{h0}},
\end{equation}
and the cost of this pair $({\bf w},\mathcal{S})$
\begin{equation}
c({\bf w},\mathcal{S}):=\sum_{h=1}^m c_hw_h,\spc\spc\spc c_h=\sum_{i=1}^N|q_{h0}-q_{hi}|^2.
\end{equation}
In the following, let $\supp({\bf w})$ denote the set of strictly positive entries of {\bf w}. Informally, we now give a translation between $N$-star transports, the feasible region of $(\ref{nstar})$, and location-fixed transportation schemes.

\begin{lemma}\label{scheme} Given $\mathcal{S}\subseteq S\times \supp(P_1)\times\ldots\times\supp(P_N)$ such that $\mathcal{T}(\mathcal{S})\neq\varnothing$:
\begin{enumerate}[i)]
\item\label{scheme_i} For each ${\bf w}\in\mathcal{T}(\mathcal{S})$, $P({\bf w},\mathcal{S})$ is a probability measure with $|\supp(P({\bf w},\mathcal{S}))|\leq|\supp({\bf w})|$.
\item\label{scheme_ii} For each ${\bf w}\in\mathcal{T}(\mathcal{S})$, there exists an $N$-star transport ${\bf y}$ between $P({\bf w},\mathcal{S})$ and $P_1,\ldots,P_N$ such that $c({\bf w},\mathcal{S})=c({\bf y})$.
\item\label{scheme_iii}  For every discrete probability measure $P$ supported on $S$ and $N$-star transport ${\bf y}$ between $P$ and $P_1,\ldots,P_N$ there exists a pair $({\bf w},\mathcal{S'})$ such that: ${\bf w}\in\mathcal{T}(\mathcal{S'})$, $P=P({\bf w},\mathcal{S'})$, and $c({\bf w},\mathcal{S'})=c({\bf y})$.
\end{enumerate}
\end{lemma}
\begin{proof}
\emph{i)} $|\supp(P({\bf w},\mathcal{S}))|\leq|\supp({\bf w})|$ is clear by definition (note that strictness of this inequality can occur if there exist non-zero $w_h,w_{h'}$ for which $q_{h0}=q_{h'0}$). To see that $P({\bf w},\mathcal{S})$ is a probability measure it suffices to show that $\sum_{h=1}^m w_h = 1$. This holds since for any $i=1,\ldots,N$ we have
\begin{equation}
\sum_{h=1}^m w_h = \sum_{k=1}^{S_i}\sum_{\substack{h,\\ q_{hi}=x_{ik}}} w_h=  \sum_{k=1}^{S_i} d_{ik} = 1,
\end{equation}
since the $P_i$ are probability measures.

\emph{ii)} For each $i=1,\ldots,N$, $j=1,\ldots,|S|$, $k=1,\ldots,S_i$ define
\begin{equation}
y_{ijk}=\sum_{\substack{h=1\\q_{h0} = x_{j} \\q_{hi} = x_{ik}}}^m w_h.
\end{equation}
Clearly $y_{ijk}\geq 0$ and it is easily checked that $\sum_j y_{ijk} = d_{ik}$ for any $i$ and $k$ and that $\sum_k y_{ijk}$ is the mass at location $x_j\in S$ in the measure $P({\bf w},\mathcal{S})$.

Hence ${\bf y}$ is an $N$-star transport between $P({\bf w},\mathcal{S})$ and $P_1,\ldots,P_N$. Further we have
\begin{align}
c({\bf y})&= \sum_{i,j,k} |x_{j}-x_{ik}|^2\cdot  y_{ijk} = \sum_{i,j,k} |x_{j}-x_{ik}|^2\cdot \sum_{\substack{h=1\\q_{h0} = x_{j} \\q_{hi} = x_{ik}}}^m w_h=\sum_{i=1}^N\sum_{j,k}\sum_{\substack{h=1\\q_{h0} = x_{j} \\q_{hi} = x_{ik}}}^m |q_{h0}-q_{hi}|^2\cdot w_h\nonumber\\
&=\sum_{i=1}^N\sum_{h=1}^m|q_{h0}-q_{hi}|^2\cdot w_h=\sum_{h=1}^mc_hw_h=c({\bf w},\mathcal{S}).
\end{align}

\emph{iii)} We note first that all of our arguments up to now not only hold for $P_i$ and $P$ being probability measures, but for any measures with total mass $0\leq r\leq 1$ that is the same for all $P_i$ and $P$. Using this fact we prove this part of the lemma for these types of measures by induction on $|\supp({\bf y})|$.\\

For $|\supp({\bf y})| = 0$, we clearly have that any $\mathcal{S}$ paired with ${\bf w} = {\bf 0}$ satifies the given conditions. So suppose $|\supp({\bf y})|>0$, then let $\mu = \min_{y_{ijk}>0} y_{ijk}$ and let $(i^*,j^*,k^*)$ be a triplet such that $y_{i^*j^*k^*} = \text{arg}\min_{y_{ijk}>0} y_{ijk}$. This implies that $d_{j^*}\geq \mu$ and so for each $i=1,\ldots,N$ there exists a $k_i$ such that $y_{ij^*k_i}\geq\mu$. In particular one can choose $k_{i^*} = k^*$ here. We then have a vector ${\bf y'}$ with $y'_{ij^*k_i} = y_{ij^*k_i}  - \mu$ and $y'_{ijk}=y_{ijk}$ otherwise. Then ${\bf y'}$ is an $N$-star transport for $P'$ to $P'_1,\ldots,P'_N$ where $P'$ is obtained from $P$ by decreasing the mass on $x_{j^*}$ by $\mu$ and each $P'_i$ is obtained from $P_i$ by decreasing the mass on $x_{ik_i}$ by $\mu$. Then $|\supp({\bf y'})|<|\supp({\bf y})|$ since $y'_{i^*j^*k^*} = 0$.

Therefore, by induction hypothesis, there exists a pair $({\bf w},\mathcal{S'})$ such that ${\bf w}\in\mathcal{T}(\mathcal{S'})$, $P'=P({\bf w},\mathcal{S'})$, and $c({\bf w},\mathcal{S'})=c({\bf y'})$ for $P'_1,\ldots,P'_N$. Let now $|\mathcal{S'}|=m$ and let $s_{m+1}=(x_{j^*},x_{1k_1},\ldots,x_{ik_i},\ldots,x_{Nk_N})$ and define $\mathcal{S}=\mathcal{S'}\cup\{s_{m+1}\}$. Then $({\bf w}^T,\mu)\in\mathcal{T}(\mathcal{S})$ and $P=P(({\bf w}^T,\mu),\mathcal{S})$ for $P_1,\ldots,P_N$. Further we have that
\begin{align}
c({\bf y})&= c({\bf y'}) + \sum_{i=1}^N c_{ij^*k_i}\mu = c({\bf y'}) + \sum_{i=1}^N |x_{j^*}-x_{ik_i}|^2\mu\nonumber\\
&=c({\bf w}^T,\mathcal{S'}) + c_{m+1}\mu = c(({\bf w},
\mu),\mathcal{S}),
\end{align}
which completes the proof by induction.  \qed

\end{proof}

We now show the existence of a transportation scheme ${\bf w^*}$ for which $|\supp({\bf w^*})|$ is provably small.

\begin{lemma}\label{sparsemin} Given a location-fixed transportation scheme $\mathcal{T}(\mathcal{S})\neq\varnothing$ for discrete probability measures $P_1,\ldots,P_N$, there exists ${\bf w^*}\in \arg\min_{{\bf w}\in\mathcal{T}(\mathcal{S})}c({\bf w},\mathcal{S})$, such that
\begin{equation}
|\supp({\bf w^*})|\leq \sum_{i=1}^N S_i - N + 1.
\end{equation}
\end{lemma}
\begin{proof} We have that $\min_{{\bf w}\in\mathcal{T}(\mathcal{S})}c({\bf w},\mathcal{S})$ is equivalent to the following LP by definition:
\begin{align}
&\min_{{\bf w}}\;\;  c({\bf w},\mathcal{S})\spc\spc\spc \nonumber\\
\sum_{\substack{h,\\ q_{hi}=x_{ik}}} w_h&= d_{ik},\spc\forall i=1,\ldots,N,\spc\forall k=1,\ldots,S_i,\nonumber\\
w_h&\geq 0,\spc\spc\spc\forall h=1,\ldots,m.
\end{align}
Thus there is a basic solution to this problem ${\bf w^*}\in\mathcal{T}(\mathcal{S})$ such that $|\supp({\bf w^*})|$ is bounded above by the rank of the matrix of the equality constraints in the first line.

Since there are $\sum_i S_i=\sum_i |\supp(P_i)|$ of these equality constraints by definition, it suffices to show that at least $N-1$ of these constraints are redundant. Let ${\bf a}_{ik}$ denote the row corresponding to the equation for some $i$ and $1\leq k\leq S_i$. Note that for a fixed $i$, $\sum_k {\bf a}_{ik}$ yields a vector of all ones, as $w_h$ appears in exactly one equation for each fixed $i$. Hence it is immediate that the row ${\bf a}_{iS_i}$ is redundant for all $i=2,\ldots,N$ since
\begin{equation}
{\bf a}_{iS_i} = {\bf 1} - \sum_{k=1}^{S_i-1} {\bf a}_{ik}=\sum_{k=1}^{S_1}{\bf a}_{1k}- \sum_{k=1}^{S_i-1} {\bf a}_{ik},
\end{equation}
where ${\bf 1}$ is the row vector of all-ones. Hence we get $N-1$ redundant rows.  \qed

\end{proof}

We are now ready to prove Theorem \ref{sparsethm}.

\begin{proof}[{of Theorem \ref{sparsethm}}] Since all barycenters are a solution to (\ref{LPequ}), there exists an $N$-star transport ${\bf y'}$ from some barycenter $\bar P'$ to $P_1,\ldots,P_N$ and $c({\bf y'})=\sum_{i=1}^N W_2(\bar{P'},P_i)^2$. By Lemma \ref{scheme} part {\em \ref{scheme_iii})}, there is some location-fixed transportation scheme $\mathcal{T}(\mathcal{S})$ for $P_1,\ldots,P_N$ and some ${\bf w'}\in\mathcal{T}(\mathcal{S})$ such that $\bar{P'}=P({\bf w'},\mathcal{S})$ and $c({\bf y'})=c({\bf w'},\mathcal{S})$. By Lemma \ref{sparsemin} there is some ${\bf w^*}\in\text{arg}\min_{{\bf w}\in\mathcal{T}(\mathcal{S})}c({\bf w},\mathcal{S})$ such that $|\supp({\bf w^*})|\leq \sum_{i=1}^NS_i - N + 1$. Now let $\bar{P} = P({\bf w^*},\mathcal{S})$, then by Lemma \ref{scheme}
\begin{equation}
|\supp(\bar P)|\leq |\supp({\bf w^*})|\leq \sum_{i=1}^NS_i - N + 1.
\end{equation}
Further, by Lemma \ref{scheme}  part {\em \ref{scheme_ii})}, there is an $N$-star transport {\bf y} between $\bar{P}$ and $P_1,\ldots,P_N$ such that
\begin{equation}
\sum_{i=1}^N W_2(\bar{P},P_i)^2 \leq c({\bf y}) = c({\bf w^*},\mathcal{S}) \leq c({\bf w'},\mathcal{S}) = c({\bf y'}) = \sum_{i=1}^N W_2(\bar{P'},P_i)^2\leq\sum_{i=1}^N W_2(\bar{P},P_i)^2,
\end{equation}
where the last inequality follows since $\bar{P'}$ is already a barycenter. Hence this chain of inequalities collapses into a chain of equalities and we see that $\bar{P}$ is the desired barycenter. \qed


\end{proof}

Finally, let us exhibit how to refine our results for discrete probability measures arising that are supported on an $L_1{\times}\dots{\times}L_d$-grid  in $\mathbb{R}^d$ that is uniform in all directions.

\begin{proof}[{of Corollary \ref{sparsegrid}}]

An $L_1{\times}\ldots{\times}L_d$-grid in $\mathbb{R}^d$ for ${\bf e_0} \in \mathbb{R}^d$ and linearly independent vectors ${\bf e_1},\dots,{\bf e_d}\in \mathbb{R}^d$ is the set $\{{\bf v}\in \mathbb{R}^d: {\bf v}={\bf e_0}+\sum\limits_{s=1}^d \frac{l_s}{L_s-1}{\bf e_s}: 0\leq l_s \leq L_s-1, l_s\in \mathbb{Z}\}$. Since by Proposition \ref{existdiscrete} we have $\supp(\bar{P})\subseteq S$, for each $x_j\in\supp(\bar{P})$ there exist $x_i={\bf e_0}+\sum\limits_{s=1}^d \frac{\alpha_{si}}{L_s-1}{\bf e_{s}}$ with $ 0\leq \alpha_{si} \leq L_s-1$ for all $i \leq N$ such that

\begin{equation}
x_j=\frac{1}{N}\sum_{i=1}^N x_i=\frac{1}{N}\sum_{i=1}^N {\bf e_0}+\frac{1}{N}\sum_{i=1}^N\sum\limits_{s=1}^d \frac{\alpha_{si}}{L_s-1}{\bf e_s} ={\bf e_0} + \sum\limits_{i=1}^N\sum\limits_{s=1}^d \frac{\alpha_{si}}{N\cdot (L_s-1)}{\bf e_s}.
\end{equation}
This tells us that $\supp(\bar{P})$ lies on the $(N(L_1-1)+1)\times \ldots \times (N(L_d-1)+1)$-grid for ${\bf e_0}$ and ${\bf e_1},\dots,{\bf e_d}$.

Since $\supp(P_i)$  lies on an $L_1{\times}\ldots{\times}L_d$-grid, the absolute bound on $|\supp(\bar{P})|$ follows immediately from Theorem \ref{sparsethm}. Since $\bar{P}$ is supported on a $(N(L_1-1)+1)\times \ldots \times (N(L_d-1)+1)$-grid, we observe a relative density of less than
\begin{equation}\frac{N(\prod\limits_{i=1}^d L_i-1) +1}{\prod\limits_{i=1}^d (N(L_i-1)+1)}\leq \frac{N\prod\limits_{i=1}^d L_i }{N^d\prod\limits_{i=1}^d (L_i-1)} =\frac{1}{N^{d-1}}\prod\limits_{i=1}^d\frac{ L_i}{ (L_i-1)},\end{equation}
in this grid, which proves the claim. \qed

\end{proof}

\section{Computations}\label{computations}

In this section we apply the computational and theoretical results developed in this paper to a hypothetical transportation problem for distributing a fixed set of goods, each month, to 9 California cities where the demand distribution changes month to month. A Wasserstein barycenter, in this case, represents an optimal distribution of inventory facilities which minimize squared distance/transportation costs totaled over multiple months. Although this data is artificially generated for purposes of exposition, the data is based on observed average high temperatures per month \cite{wiki:climate}. All the source code used in this section is publicly available through the on-line repository \url{https://github.com/EthanAnderes/WassersteinBarycenterCode}

The probability measures used in this example are defined on $\Bbb R^2$ and are denoted $P_{\text{dec}}$, $P_{\text{jan}}$, $P_{\text{feb}}$, $P_{\text{mar}}$, $P_{\text{jun}}$, $P_{\text{jul}}$, $P_{\text{aug}}$ and $P_{\text{sep}}$ to correspond with $8$ months of the year (scaling up to $12$ months, while not intractable, imposes unnecessary computation burdens for computational reproducibility). The support of each distribution is given by the longitude-latitude coordinates of the following $9$ California cities: Bakersfield, Eureka, Fresno, Los Angeles, Sacramento, San Bernardino, San Francisco, San Jose and  South Lake Tahoe. The mass distribution assigned to each $P_{\text{dec}},\ldots, P_{\text{sep}}$ is computed in two steps. The first step calculates 
\[(\text{population in city $C$})\times(\text{average high temp for month $M$ - $72^o$})^2\]
 for each city $C$ and each month $M$. The second step simply normalizes these values within each month to obtain $8$ probability distributions defined over the same $9$ California cities. Figure \ref{inputFigs} shows $P_{\text{feb}}$, $P_{\text{mar}}$, $P_{\text{jun}}$ and $P_{\text{jul}}$.

Let $\bar P$ denote an optimal Wasserstein barycenter as defined by Equation (\ref{two}). Proposition  \ref{existdiscrete} and Theorem \ref{sparsethm} both give bounds on the support of $\bar P$ uniformly over rearrangement of the mass assigned to each support point in $P_{\text{dec}},\ldots, P_{\text{sep}}$. Proposition \ref{existdiscrete} gives an upper bound for  $\text{supp}(\bar P)$ in the form of a finite covering set which guarantees that finite dimensional linear programing can yield all possible optimal $\bar P$ (see (\ref{LPequ})).  Conversely, Theorem \ref{sparsethm} gives an upper bound for the magnitude $|\text{supp}(\bar P)|$ which is additionally uniform over rearrangement of the locations of the support points in $P_{\text{dec}},\ldots, P_{\text{sep}}$.

In the implementation presented here we use the modeling package {\em JuMP} \cite{LubinDunningIJOC} which supports the open-source {\em COIN-OR} solver {\em Clp} for linear programming within the language {\em Julia} \cite{beks-14}.  The set $S$, defined in (\ref{centers}), covers the support of $\bar P$ and is shown in the rightmost image of Figure \ref{barycenterFigs}. A typical {\em stars and bars} combinatorial calculation yields $|S|=\binom{9+8-1}{9-1}=\binom{16}{8}=12870$. The corresponding LP problem for $\bar P$ therefore has $939510$ variables with $103032$ linear constraints. On a 2.3 GHz Intel Core i7 MacBook Pro a solution was reached after $505$ seconds (without using any pre-optimization step). The solution is shown  in the leftmost image of Figure \ref{barycenterFigs}.
Notice that Theorem \ref{sparsethm} establishes an upper bound of $65 = 9 \cdot 8 - 8 + 1$ for $|\text{supp}(\bar P)|$. The LP solution yields $|\text{supp}(\bar P)| = 63$. Not only does this give good agreement with the sparsity bound from Theorem \ref{sparsethm} but also illustrates that Wasserstein barycenters are very sparse with only $0.5\%$ of the possible support points in $S$ getting assigned non-zero mass.

In Figure \ref{transportFigs} we illustrate Theorem \ref{nomasssplit} which guarantees the existence of pairwise optimal transport maps from $\bar P$ to each $P_{\text{dec}},\ldots, P_{\text{sep}}$ which do not split mass. The existence of these discrete non-mass-splitting optimal transports is a special property of $\bar P$. Indeed, unless special mass balance conditions hold, there will not exist any transport map (optimal or not) between two discrete probability measures. The implication for this example is that {\em all} the inventory stored at a barycenter support point will be optimally shipped to exactly one city each month. Moreover, since the transportation displacements must satisfy Theorem \ref{nomasssplit} {\em \ref{nomasssplit_iii})} each city is at the exact center of its $8$ monthly transportation plans.

\section*{Acknowledgements}
SB acknowledges support from the Alexander-von-Humboldt Foundation. JM acknowledges support from a UC-MEXUS grant. EA acknowledges support from NSF CAREER grant DMS-1252795. The authors would like to thank  Jes\'us A. De Loera,  Hans M\"uller and Jonathan Taylor for many enlightening discussions on Wasserstein barycenters.
\newpage

\bibliography{barycenters_literature}
\bibliographystyle{plain}

\end{document}